\documentclass[11pt,letterpaper]{article}
\usepackage{amsmath,amssymb,amsfonts,amsthm,mathrsfs, url}
\usepackage{graphicx,fullpage}
\usepackage[usenames,dvipsnames]{color}

\usepackage[margin=1in]{geometry}

\usepackage[normalem]{ulem}

\theoremstyle{plain}
\newtheorem{thm}{Theorem}
\newtheorem{lem}[thm]{Lemma}
\newtheorem{cor}[thm]{Corollary}

\theoremstyle{definition}
\newtheorem{defn}[thm]{Definition}
\newtheorem{remark}[thm]{Remark}

\newcommand{\xc}[1]{{\color{magenta}{\bf [ XP: }#1 {\bf ]}}}

\newcommand{\nat}{\ensuremath {\mathbb N} }

\newcommand{\mbf}[1] {\text{\boldmath$#1$}}
\newcommand{\remove}[1] {}

\newcommand{\ex} {{\bf E}}
\newcommand{\pr} {{\bf Pr}}
\newcommand{\var} {{\bf Var}}

\newcommand{\eps}{\varepsilon}

\def\st{\colon\,}
\def\NN{{\mathbb N}}
\def\OR{{\rm OR}}

\def\olr{\mathring{R}}

\def\prk{P_r^{(k)}}
\def\VEC#1#2#3{#1_{#2},\ldots,#1_{#3}}
\def\FL#1{\lfloor{#1}\rfloor}
\def\CL#1{\lceil{#1}\rceil}
\def\CH#1#2{\binom{#1}{#2}}
\def\top{{\mbf \Lmb}}

\def\Yup{\overleftarrow{Y}}
\def\PE#1#2#3{\prod_{#1=#2}^{#3}}

\def\tow{{\rm tow}}
\def\srt{\hat R_t}
\def\olsrt{\tilde R_t}
\def\hR{\hat R}

\def\esub{\subseteq}
\def\bz{{\bf V}}
\def\CHnj{\CH{[n]}j}
\def\CHnk{\CH{[n]}k}
\def\CHnjp{\CH{[n]}{j+1}}
\def\prkl{P_r^{k,\ell}}
\def\prkkm{P_r^{k,1}}
\def\prkkk{P_r^{k,k}}
\def\Lmb{\mathbf{\Lambda}}
\def\FR{\frac}

\title{On-line size Ramsey number for monotone $k$-uniform ordered paths with
uniform looseness}
\author{
Xavier P\'erez-Gim\'enez\thanks{Department of Mathematics, University of 
Nebraska--Lincoln, Lincoln, NE, USA, e-mail: \texttt{xperez@unl.edu}}
\and Pawe{\l} Pra{\l}at\thanks{Department of Mathematics, Ryerson University,
Toronto, ON, Canada, e-mail: \texttt{pralat@ryerson.ca}; Supported in part by
NSERC and Ryerson University.}
\and Douglas B. West\thanks{Departments of Mathematics, Zhejiang Normal
University, Jinhua, China, and University of Illinois, Urbana, IL, USA, e-mail:
\texttt{dwest@math.uiuc.edu}; Research supported by Recruitment Program of
Foreign Experts, 1000 Talent Plan, State Administration of Foreign Experts
Affairs, China.}
}

\date{\today}

\begin{document}
\maketitle

\vspace{-2pc}
\begin{abstract}
An {\it ordered hypergraph} is a hypergraph $H$ with a specified linear
ordering of the vertices, and the appearance of an ordered hypergraph
$G$ in $H$ must respect the specified order on $V(G)$.  In
on-line Ramsey theory, Builder iteratively presents edges that Painter
must immediately color.  The {\it $t$-color on-line size Ramsey number}
$\olsrt(G)$ of an ordered hypergraph $G$ is the minimum number of edges Builder
needs to play (on a large ordered set of vertices) to force Painter using $t$
colors to produce a monochromatic copy of $G$.  The {\it monotone tight path}
$P_r^{(k)}$ is the ordered hypergraph with $r$ vertices
whose edges are all sets of $k$ consecutive vertices.

We obtain good bounds on $\olsrt(\prk)$. Letting $m=r-k+1$ (the number of edges
in $\prk$), we prove $m^{t-1}/(3\sqrt t)\le\olsrt(P_r^{(2)})\le tm^{t+1}$.  For
general $k$, a trivial upper bound is $\CH Rk$, where $R$ is the least number of
vertices in a $k$-uniform (ordered) hypergraph whose $t$-colorings all contain
$\prk$ (and is a tower of height $k-2$).  We prove
$R/(k\lg R)\le\olsrt(\prk)\le R(\lg R)^{2+\epsilon}$, where $\epsilon$ is
any positive constant and $t(m-1)$ is sufficiently large.  Our upper
bounds improve prior results when $t$ grows faster than $m/\log m$.
We also generalize our results to $\ell$-loose monotone paths, where each
successive edge begins $\ell$ vertices after the previous edge.
\end{abstract}

\baselineskip 16pt
\section{Introduction}
Ramsey theory studies the occurrence of forced patterns in colorings.  We say
that $H$ {\it forces} $G$ and write $H\to_t G$ when every $t$-coloring of the
elements of $H$ contains a monochromatic copy of $G$.  In this paper $H$ and
$G$ are $k$-uniform hypergraphs, we color the edges of $H$, and $t\ge2$.
Ramsey's Theorem~\cite{R} implies $K_n^{(k)}\to_t G$ when $n$ is sufficiently
large, where $K_n^{(k)}$ denotes the complete $k$-uniform hypergraph with $n$
vertices.  Our problem involves several variations on this.

For any monotone parameter, we can study its least value on the (hyper)graphs
that force $G$.  Aside from the number of vertices (the classical problem), the
most-studied parameter for this is the number of edges, yielding the
{\it size Ramsey number} (proposed in~\cite{EFRS}, with early work surveyed
in~\cite{FS}).  For example, Beck~\cite{B1} solved a problem of Erd\H{o}s by
showing that the $2$-color size Ramsey number of the path $P_n$ is linear in
$n$; after improvements in \cite{Bo,DP1,Let}, the current best upper bound is
$74n$ by Dudek and Pra{\l}at~\cite{DP2}.

Another direction considers an ordered version of hypergraphs.  An
\emph{ordered hypergraph} is a hypergraph on a linearly ordered vertex set.
In the ordered sense, $H$ is a \emph{subhypergraph} of $H'$ if $H'$ contains a
copy of $H$ with the vertices appearing in the specified order.  Since a
complete ordered hypergraph contains all ordered hypergraphs with that many
vertices, Ramsey's Theorem also holds in the ordered sense.  That is, for
an ordered $k$-uniform hypergraph $G$, there exist ordered $k$-uniform
hypergraphs $H$ such that $H\to_t G$ (every $t$-coloring of $E(H)$ contains a
monochromatic copy of $G$ in the ordered sense).  Thus Ramsey numbers
and size Ramsey numbers for ordered hypergraphs are also well-defined.  Such
problems have been studied in \cite{BCKK,CFLS,CS,FPSS,MSW,MS,Mub,MubS}.

An ``on-line'' version of Ramsey theory is a game between {\it Builder} and
{\it Painter}, introduced by Beck~\cite{B3} and by Kurek and Rucinski~\cite{KR}.
In each round, Builder presents an edge that Painter must color.  When the
Ramsey value for the target $G$ is well-defined, Builder can force a
monochromatic copy of $G$ by presenting all edges of some $H$ such that
$H\to_t G$.  However, Builder may be able to use Painter's choices to force $G$
to appear sooner.  On-line Ramsey problems have been studied for the number of
edges \cite{Con,CD,CDLL,GHK,GKP,KR,Pra1,Pra2}, the genus \cite{GHK,KK,Pet}, and
the maximum degree \cite{BGKMSW,KW,Ro1,Ro2}.  The number of edges is the number
of rounds (the length of the game) and hence is the natural parameter.  It is
so natural that the on-line size Ramsey number has confusingly also been called
just the on-line Ramsey number.  Easy arguments imply that the $2$-color
on-line size Ramsey number of the path $P_n$ is at least $2n-3$ and at most
$4n-7$ (\cite{GKP}).

We study the on-line size Ramsey number of monotone tight paths.  Let $[r]$
denote $\{1,\ldots,r\}$.  The {\it monotone tight path} is the $k$-uniform
ordered hypergraph $P_r^{(k)}$ with vertex set $[r]$ whose edges are all sets
of $k$ consecutive vertices.  The vertex analogue was studied and applied
in~\cite{DLR,FPSS,MSW,MS}.

For Ramsey problems, we follow a common practice of adding a circumflex accent
($\hR$) to indicate the size Ramsey number.  Several recent papers use a tilde
accent to indicate the on-line version of the size Ramsey number (a circular
accent $\olr$ has been used with on-line versions of other parameter Ramsey
numbers).  These choices free the subscript for the number of colors.  For
ordered Ramsey numbers, $\OR$ was used in \cite{MSW}, but it seems natural to
use the same notation as in the classical problem when it is understood that
the target and host are ordered hypergraphs (see \cite{Con,Mub,MubS}).  Thus we
use $\olsrt(\prk)$ for the $t$-color on-line size Ramsey number of the monotone
tight path $\prk$.

Our results and proofs are motivated by the characterization of the $t$-color
off-line vertex Ramsey number of $\prk$ by Moshkovitz and Shapira~\cite{MS}
(see~\cite{MSW} for an exposition and alternative presentation of the proof).
Henceforth let $m$ be the number of edges in $\prk$; note that $m=r-k+1$.
The arguments and bounds are stated more cleanly in terms of $m$.  Let $Q_1$ be
the poset (partially ordered set) consisting of $t$ disjoint chains of size
$m-1$.  For $j>1$, let $Q_j$ be the poset consisting of all the down-sets in
$Q_{j-1}$, ordered by inclusion.  The bounds on $|Q_k|$ follow inductively.

\begin{thm}[Moshkovitz and Shapira~\cite{MS}]\label{MSthm}
$R_t(\prk)=|Q_k|+1$.  Furthermore,
\[
\tow_{k-2}(m^{t-1}/2\sqrt t)\le |Q_k|\le \tow_{k-2}(2m^{t-1}),
\]
where $m=r-k+1$ and $\tow_h(x)$ equals $x$ when $h=0$ and
$2^{\tow_{h-1}(x)}$ when $h\ge1$.
\end{thm}

This result immediately implies $\olsrt(\prk)\le \CH{|Q_k|+1}k$, since
$\CH{|Q_k|+1}k$ is the number of edges in $K_{|Q_k|+1}^{(k)}$.  Building on
ideas used in the exposition of this proof in~\cite{MSW}, we present a strategy
for Builder proving an upper bound of $|Q_k|(\lg|Q_k|)^{2+\epsilon}$, where
$\lg$ is the base-$2$ logarithm and $\epsilon$ is any positive constant.  Our
Painter strategy for the lower bound yields roughly the same lower bound as in
Theorem~\ref{MSthm}.  Hence our upper and lower bounds on $\olsrt(\prk)$ are
towers of the same height.

The arguments for the upper and lower bound generalize trivially to the
non-diagonal case $\olsrt(P_{r_1}^{(k)},\dots,P_{r_t}^{(k)})$, where Builder
seeks to force a copy of $P_{r_i}^{(k)}$ in color $i$ for some $i$.  Simply let
$Q_1$ be the disjoint union of $t$ chains such that the $i$th chain has
$r_i-k+1$ elements.

Fox, Pach, Sudakov, and Suk~\cite{FPSS} considered a game with a more
restricted Builder, which was introduced by Conlon, Fox, and Sudakov~\cite{CFS}.
Builder can only introduce a new vertex at the end of the ordering and present
some edges joining the newest vertex to earlier vertices.  Painter colors them
immediately.  Our Builder can simulate this game, so the optimal value $f_t(m)$
in their game is at least $\olsrt(P_{m+1}^{(2)})$.  For constant $t$ (and here
$k=2$), Fox et al.~\cite{FPSS} proved
\[
\frac{t-1-o(t)}{\log t}m^t\log m \le f_t(m)
\le \left(1+\frac{t-1}{\log(1+1/(t-1))}\log(m+1)\right)(m^t+1).
\]
Since their Builder is weaker, their lower bound is naturally larger than ours;
neither result implies the other.  For large $t$ (growing faster than
$m/\log m$), our upper bound is smaller than theirs, but for constant $t$
their upper bound is better.

They also studied the $k$-uniform version of their game, where their objective
was to obtain an upper bound on the vertex Ramsey number of the monotone tight
path in terms of the length of their game.  Since $R_t(\prk)=|Q_k|+1$
by~\cite{MS}, in their game also Builder must use more than $|Q_k|$ vertices to
end the game.

Indeed, if Painter knows that the game is being played by their Builder,
meaning that vertices will only be introduced from left to right, then Painter
can use our strategy (in the general $k$-uniform case) with a supply of $Q_k$
vertices (treating them as described in Section~\ref{sec:hyp}), achieving
$|Q_k|/k$ as a lower bound against their Builder.  Similarly, when the vertices
are known initially (that is, in the off-line setting), our Painter strategy
also implies that any hypergraph forcing $\prk$ has more than $|Q_k|$ vertices,
thus yielding the lower bound $R_t(\prk)>|Q_k|$ in Theorem~\ref{MSthm}.
A closer look at the upper bound strategy for Builder also yields the 
upper bound $R_t(\prk)\le |Q_k|+1$.  The ideas in our proof are similar to the
ideas in the proofs in~\cite{MS} and~\cite{MSW}.

Our proofs also generalize easily to describe the Ramsey number of the monotone
{\em $\ell$-loose $k$-uniform path} $\prkl$ for $1\le\ell\le k$.  Here each edge
consists of $k$ consecutive vertices, and two consecutive edges have $k-\ell$
common vertices.  (In particular, $r=k+\ell(m-1)$ when there are $m$ edges.)
Note that $\prk=\prkkm$, while $\prkkk$ is a $k$-uniform matching in which each
edge ends before the next edge begins in the vertex ordering.  Let
$h=\CL{k/\ell}$.  Our arguments for the on-line version of the problem yield 
$R_t(\prkl)=\ell|Q_h|+s$, where $s=k-(h-1)\ell$.  This formula was obtained
earlier by Cox and Stolee~\cite{CS}, expressed in different notation.
They gave a separate argument for the case $\ell=k$ (matchings), though
this formula applies to both.

In the last section we discuss an off-line version of this problem for
directed graphs and hypergraphs, related to results of 
Ben-Eliezer, Krivelevich, and Sudakov~\cite{BKS}.

\section{On-line scenario: The graph case ($k=2$)}

The game ends when Builder forces Painter to produce a monochromatic monotone
tight path with $m$ edges.  For clarity and because the numerical bounds are
somewhat tighter in this case, we first consider the case $k=2$.  For the
monotone path, $R_t(P_r^{(2)})=m^t+1$.  The trivial upper bound is
$\CH{R_t(P_r^{(2)})}2$, but our upper bound is not much larger than
$R_t(P_r^{(2)})$.  Like the result of~\cite{MS}, it is motivated by the short
proof due to Seidenberg~\cite{Se} of the Erd\H{o}s--Szekeres Theorem~\cite{ES}
on monotone subsequences.

\begin{thm}\label{k=2}
For $m=r-1$ with $r\ge3$, always
$m^{t-1}/(3\sqrt t)\le\olsrt(P_r^{(2)})\le tm^{t+1}$.
\end{thm}

\begin{proof}
Let $M = \{0,1,\ldots,m-1\}$.
Given $a=(\VEC a1t)\in M^t$, let $|a|=\sum a_i$.

\smallskip
{\bf Upper bound} {\it (Builder strategy):}
Builder uses $m^t+1$ vertices, viewed as ordered from left to right.  At
any time, all vertices are labeled with vectors in $M^t$, where the $i$th
coordinate of the label for $v$ is the number of edges in the longest monotone
path in color $i$ that ends at $v$.  All labels are initially the all-$0$
vector.  Let $\top$ denote the ``top'' vector in $M^t$; its components all
equal $m-1$.

Builder seeks to produce label $\top$ at one of the first $m^t$ vertices, after
which playing the edge from this vertex to the last (rightmost) vertex wins the
game no matter what color Painter gives it.  If no two vertices among the first
$m^t$ have the same label, then all labels occur, including $\top$.  

Otherwise, some vertices $u$ and $v$ have the same label, say with $u$ before
$v$.  These vertices cannot yet be adjacent, since their labels would then
differ in the coordinate for the color of $uv$.  Builder plays $uv$.  The label
for $v$ increases in the coordinate for the color Painter uses on $uv$.

On each round, the label for the second vertex of the edge played increases by
$1$ in some coordinate.  To avoid reaching $\top$ or reaching $m$ in any
coordinate, each label must increase fewer than $(m-1)t$ times.  By the
pigeonhole principle, within $m^t[(m-1)t-1]+1$ rounds some label reaches
$\top$, and the next play wins.  Note that $m^t[(m-1)t-1]+1<tm^{t+1}$.

\smallskip
{\bf Lower bound} {\it (Painter strategy):}
Let $B =\{a\in M^t\st |a|=\FL{(m-1)t/2}\}$.  Until Builder uses more than
$|B|$ vertices, Painter can assign different labels from $B$ to all vertices
used.  These labels remain unchanged throughout the game.  Let $a(v)$ denote
the label assigned by Painter to $v$, with $a(v)=(a_1(v),\dots,a_t(v))$.  When
Builder plays an edge $uv$ with $u$ before $v$, Painter gives it a color $i$
such that $a_i(v)>a_i(u)$.  Such a coordinate exists, since $a(u)\ne a(v)$ and
$|a(u)|=|a(v)|$.

Choosing colors in this way maintains for each vertex $w$ the property that 
every monotone path in color $i$ arriving at vertex $w$ has at most $a_i(w)$
edges.  This holds since along a monotone path in color $i$ the $i$th
coordinate of the label strictly increases with each step.  Since
$a(w)\in M^t$, no monochromatic monotone path has $m$ edges.  Since using more
than $|B|$ vertices requires playing more than $|B|/2$ edges, Painter can
survive at least $|B|/2$ rounds without creating a monochromatic monotone path
with $m$ edges.

The elements of $M$ are the elements of $Q_2$, and $B$ is a middle level.
Using Chebyshev's Inequality and the pigeonhole principle,
Moshkovitz and Shapira~\cite{MS} showed $|B|\ge \frac23 m^{t-1}/\sqrt t$.
\end{proof}

\begin{remark}
It is well known by many arguments that $B$ is a largest level in $Q_2$.
(For example, the product of chains is a symmetric chain order, the convolution
of symmetric log-concave sequences is symmetric and log-concave, explicit
injections map one level to the next toward the middle, etc.)  Since
$|M^t|=m^t$ and there are $(m-1)t+1$ levels, we thus have $|B|> m^{t-1}/t$ by
the pigeonhole principle alone.

Using the Chernoff bound instead of Chebyshev's Inequality
in the argument in~\cite{MS}, we can improve
the lower bound on $|B|$ to $0.7815987 m^{t-1}/\sqrt t$.  The value of $|B|$
was also studied by Alekseev~\cite{Al}.  A special case is that when
$m\in o(e^t/\sqrt t)$, the value of $|B|$ is asymptotic to
$m^{t-1}/\sqrt{\pi t/6}$.

For the non-diagonal case, with $m_i$ being the forbidded
length in color $i$, the argument yields
\[
\frac{\prod m_i}{2\sum m_i}\le 
\olsrt(P_{r_1}^{(2)},\dots,P_{r_t}^{(2)})\le
\sum m_i\prod m_i.
\]
Here the pigeonhole argument for the size of the largest antichain in $Q_2$
gives the lower bound on $|B|$.  Again Chebyshev's Inequality can be used to
improve it somewhat, but the resulting formula is more complicated.

Our lower bound remains valid against a stronger Builder.  Suppose Builder
can present any directed graph in seeking a monochromatic directed path,
instead of only presenting edges directed from lower to higher vertices.  The
strategy for Painter establishes the same lower bound, where ``an edge $uv$
with $u$ before $v$'' becomes ``an edge directed from $u$ to $v$''.  This 
works because the labels for vertices are incomparable.  We will return to
the digraph problem in the last section.
\end{remark}

\section{On-line scenario: The hypergraph case}\label{sec:hyp}

For the $k$-uniform monotone tight path, the flavor of the arguments extends
that of the graph case, but the details are more delicate.
As described in the introduction, let $Q_1$ be the poset consisting of $t$
disjoint chains of $m-1$ elements each.  The $i$th chain is associated with 
color $i$.  For $j>1$, the poset $Q_j$ consists of the down-sets in $Q_{j-1}$,
ordered by inclusion.  The arguments are the same for the non-diagonal case,
with the $i$th chain in $Q_1$ consisting of $m_i-1$ elements, where
$m_i=r_i-k+1$.

We will first study the upper bound.  Let $G$ denote the current hypergraph of
edges played by Builder and colored by Painter.  In the strategy for Builder
used to prove the upper bound, Builder will confine play to a fixed vertex set
$[n]$, where $[n]=\{1,\ldots,n\}$, under the usual order on $\NN$.
Given a set $Y\esub[n]$, let $Y^+$ be the set obtained from $Y$ by deleting
the first vertex, and let $Y^-$ be the set obtained from $Y$ by deleting
the last vertex.  Let $\CH{[n]}j$ denote the family of $j$-element subsets
of $[n]$.  We define functions $\VEC gk1$ such that $g_j\st\CHnj\to Q_{k-j+1}$,
except that $g_k$ is defined only on the $k$-sets that are actual edges of $G$.
These functions will be used in Builder's strategy while $G$ has no
monochromatic $\prk$.

\begin{defn}\label{labelfcn}
For $Y\in E(G)$, if $Y$ has color $i$ and
the longest monochromatic tight path with last edge $Y$ has $p$ edges, then
let $g_k(Y)$ be element $p$ on the $i$th chain in $Q_1$.
For $Y\in \CHnj$ with $j<k$, let $\Yup=\{Z\in \CHnjp\st Z^+=Y\}$; call the
elements of $\Yup$ the {\it precursors} of $Y$.
Given that $g_{j+1}$ has been defined, for $Y\in \CHnj$
define $g_j(Y)$ as follows:

\medskip
\hfill
{$g_j(Y)$ is the downset in $Q_{k-j}$ generated by
$\{g_{j+1}(Z)\st Z\in \Yup\}$.}
\hfill
\medskip

\noindent
Being a downset in $Q_{k-j}$, by definition $g_j(Y)\in Q_{k-j+1}$.
\end{defn}

\begin{defn}\label{follows}
Given $Y_1,Y_2\in E(G)$, say that $Y_2$ {\it follows} $Y_1$ if $Y_1^+=Y_2^-$.
For $Y_1,Y_2\in \CHnj$ with $j<k$, say that
$Y_2$ {\it follows} $Y_1$ if

\qquad~(A) $Y_1^+=Y_2^-$ and

\qquad~(B) for each maximal element $w$ of $g_j(Y_1)$, the $(j+1)$-set
$Y_1\cup Y_2$ 

\qquad\qquad follows some precursor $Z_1$ of $Y_1$ such that $g_{j+1}(Z_1)=w$.
\end{defn}

Note that (B) in Definition~\ref{follows} holds trivially when $g_j(Y_1)$
is empty.  Since a precursor $Z_2$ of $Y_2$ following a
precursor $Z_1$ of $Y_1$ requires $Z_2^-=Z_1^+=Y_1$, the set $Y_1\cup Y_2$ is
the only precursor of $Y_2$ that can follow a precursor of $Y_1$.  When $Y_2$
follows $Y_1$, the set $Y_1\cup Y_2$ is a set $Z$ such that $Z^-=Y_2$ and
$Z^+=Y_2$.  Our strategy for Builder is based on the following crucial property
of $g_j$.

\begin{lem}\label{eq:star}
{If $Y_2$ follows $Y_1$ in $\CHnj$, then $g_j(Y_1) \ngeq g_j(Y_2)$
in $Q_{k-j+1}$.}
\end{lem}
\begin{proof}
The proof is by induction on $k-j$.  For $j=k$, if $Y_2$ follows $Y_1$ in
$E(G)$, then either $Y_1$ and $Y_2$ have the same color, in which case
$g_k(Y_2)>g_k(Y_1)$ in $Q_1$, or they have different colors, in which case
$g_k(Y_1)$ and $g_k(Y_2)$ are incomparable in $Q_1$.  In either case,
$g_k(Y_1)\ngeq g_k(Y_2)$.

For $j<k$, suppose that the claim holds for $j+1$.  Given that $Y_2$ follows
$Y_1$ in $\CHnj$, let $Z=Y_1\cup Y_2\in \CHnjp$.  If $Y_1$ has no precursors
(that is, $g_j(Y_1)$ is empty), then the statement is trivially true since $Z$
is a precursor of $Y_2$ and thus $g_j(Y_2)$ is nonempty.  Otherwise, let $w$ be
a maximal element of $g_j(Y_1)$.  Since $Y_2$ follows $Y_1$, by definition $Z$
follows some $Z_1\in \Yup_1$ with $g_{j+1}(Z_1)=w$.  By the hypothesis for
$j+1$, we have $w=g_{j+1}(Z_1)\ngeq g_{j+1}(Z)$ for all such $Z_1$.  Since this
holds for all $w$ that are maximal in $g_j(Y_1)$, the label $g_{j+1}(Z)$ does
not lie in the downset generated by the precursors of $Y_1$ (which by
definition is $g_j(Y_1)$).  However, since $Z\in\Yup_2$, the label $g_{j+1}(Z)$
does lie in $g_j(Y_2)$.  Hence as downsets in $Q_{k-j}$, the family $g_j(Y_2)$
is not contained in the family $g_j(Y_1)$, which means
$g_j(Y_1)\ngeq g_j(Y_2)$ as elements of $Q_{k-j+1}$.
\end{proof}

The inductive definition of ``follows'' facilitates Lemma~\ref{eq:star}.  To
simplify the presentation of the Builder's strategy, we provide a more explicit
description of what ``$Y_2$ follows $Y_1$'' guarantees.

\begin{defn}\label{UY}
For a $j$-set $Y$ with $j<k$ or an edge $Y\in E(G)$, we form a tree $U(Y)$.  The
nodes of the tree are elements of the posets $\VEC Q{k-j+1}1$ occurring as
labels.
The root of $U(Y)$ is the label $g_j(Y)\in Q_{k-j+1}$.  For any node $w$
in $U(Y)$ that is in $Q_i$ for $i>1$, its children are the maximal elements of
the downset in $Q_{i-1}$ that is $w$.  The process iterates until we reach
elements of $Q_1$ as the leaves of $U(Y)$.

An {\it instance} of $U(Y)$ associates vertex sets to the nodes.  Associated to
the root of $U(Y)$, which has label $g_j(Y)\in Q_{k-j+1}$, is the set $Y$.
To a non-root node $w\in Q_i$ whose parent in $U(Y)$ is $z\in Q_{i+1}$ and has
associated $(k-i)$-set $Z$, we associate a precursor $Z'$ of $Z$ such that
$g_i(Z')=w$; note that $Z'$ is a $(k-i+1)$-set.  Iteratively, we choose
associated sets moving away from the root.  Since the leaves are labels in
$Q_1$, their associated sets are $k$-sets: that is, edges.  From the definition
of $\VEC g1j$, every such tree $U(Y)$ has at least one such instance.
\end{defn}

\begin{lem}\label{followtree}
A $j$-set $Y_2$ follows a $j$-set $Y_1$ if and only if $Y_1^+=Y_2^-$ and
there is an instance of $U(Y_1)$ such that for every edge $W$ associated with a
leaf, replacing the first vertex of $W$ with the last vertex of $Y_2$ yields an
edge $Z$ in $G$.
\end{lem}
\begin{proof}
As we move away from the root node in $U(Y_1)$, with each step the precursors
get larger by adding vertices at the left.  We construct the needed instance of
$U(Y_1)$ by associating labels along each path from the root.
Given that $Y_2$ follows $Y_1$, let $Z_1=Y_1\cup Y_2$.  Note that $Z_1$ arises
from $Y_1$ by adding the last vertex $y$ from $Y_2$.  By the definition of
$Y_2$ following $Y_1$, the set $Z_1$ is required to be a $(j+1)$-set that, for
each child $w_1$ of the root of $U(Y_1)$, follows some precursor $W_1$ of $Y_1$
that has label $w_1$.
This selects $W_1$ as a $(j+1)$-set to associate with $w_1$ in the instance of
$U(Y_1)$ we are building.  Repeating this observation along a path from the
root to a leaf of $U(Y_1)$, we obtain successively larger sets $\VEC Z1{k-j}$
that follow corresponding sets $\VEC W1{k-j}$ associated with the nodes along
the path.  Each $Z_i$ is obtained by deleting the smallest element of $W_i$ and
adding $y$.  Finally, $Z_{k-j}$ is an edge following an edge $W_{k-j}$
associated with the leaf at the end of the path.  We obtain such an edge
$Z_{k-j}$ for each leaf.
\end{proof}

\begin{remark}\label{bottom}
For $Y\in\CH{[n]}j$, if no precursor of $Y$ has a defined label, then the
downset generated by $\Yup$ is empty, and $g_j(Y)$ is the bottom element of
$Q_{k-j+1}$.  This occurs for a $(k-1)$-set whose precursors all are not edges
of $G$ and for any $j$-set with first vertex $1$ (it has no precursors).

Each of $\VEC Q2k$ has one element of rank $0$, which is the empty downset in
the previous poset.  Also each of $\VEC Q3k$ has one element of rank $1$, which
is the downset of size $1$ consisting of the bottom element of the previous
poset.  Inductively, ranks $0$ through $j-2$ of $Q_j$ form a single chain with
one element of each rank.  For $0\le i\le j-2$, let $\bz_j^i$ be the element of
rank $i$ in $Q_j$.

With vertex set $[n]$ before any edges have been played, all $k$-sets have
undefined labels.  Hence the label of each $(k-1)$-set is $\bz_2^0$.
The label of a $j$-set with least element $1$ is $\bz_{k-j+1}^0$.
A $j$-set $Y$ with least element $2$ has one precursor, with label
$\bz_{k-j}^0$, so $g_j(Y)=\bz_{k-j+1}^1$.
Inductively, for $j<k$, a $j$-set $Y$ with least element $i$ has initial label
$\bz_{k-j+1}^{i-1}$ if $i\le k-j$ and label $\bz_{k-j+1}^{k-j-1}$ if $i>k-j$.
In particular, for the crucial case $j=1$, the initial label of the vertex $i$
is $\bz_k^{i-1}$ for $i\le k-1$ and $\bz_k^{k-2}$ for $i>k-1$.
\end{remark}

Our upper bound for general $k$ is also valid for $k=2$, but in that case
Theorem~\ref{k=2} provides a stronger bound.  For $k=3$ our bound is a bit
weaker than for larger $k$, which introduces some complication in the inductive
proof.  The combinatorial bound obtained first is valid for all $k,m,t$, but
the bound in terms of $|Q_k|$ alone requires $tm$ (or equivalently $|Q_1|$) to
be sufficiently large.

\begin{thm}\label{main}
For $k,m,t\in\NN$ with $t,m\ge2$ and $r=k+m-1$.
\[
\olsrt(\prk)\le |Q_k|\cdot |Q_{k-1}|\PE i1{k-1}a_i,
\]
where $a_i$ is the size of the largest antichain in $Q_i$.
Moreover, for any positive constant $\epsilon$,
\[
|Q_3|\cdot |Q_{2}| a_2 a_1 \le |Q_3| (\lg|Q_3|)^{2+\FR1{t-1}+\eps}
\qquad\text{and}\qquad
|Q_k|\cdot |Q_{k-1}|\PE i1{k-1}a_i \le |Q_k| (\lg|Q_k|)^{2+\eps}
\quad
\text{(for $k\ge 4$)}
\]
when {$tm$} is sufficiently large in terms of $\eps$.
\end{thm}
\begin{proof}
We give a strategy for Builder.  Let $n=|Q_k|+1$.  Builder plays on the fixed
ordered vertex set $[n]$, numbered from left to right.
After each round the functions $\VEC gk1$ are defined as in
Definition~\ref{labelfcn} for the hypergraph played so far.  Let $\Lmb_j$ be
the unique top element in $Q_j$, for $2\le j\le k$.  Builder seeks a vertex $z$
in $[n]-\{n\}$ with $g_1(z)=\Lmb_k$.  This vertex $z$ must have a precursor
$\{y,z\}$ with label $\Lmb_{k-1}$, since $\Lmb_k$ is the downset in $Q_{k-1}$
that is all of $Q_{k-1}$.  Iterating, some $(k-1)$-set $Y$ ending at $z$ has
label $\Lmb_2$.  Since $\Lmb_2=(m-1,\dots,m-1)$, in each color some precursor
of $Y$ is the edge ending a path of $m-1$ edges.  Builder then plays the edge
${Y\cup\{n\}}$ to win.

Builder plays to force Painter to produce such a vertex $z$.  Before any edges
are played, the labels are as described in Remark~\ref{bottom}.  The labels
of the first $k-1$ vertices never change, since no edge can be played ending
at one of those vertices.  All vertices from $k-1$ to $n$ initially have the
same label, with rank $k-2$ in $Q_k$.

Playing an edge in the game creates a label for that edge.  The label of an
existing edge stays the same or moves upward on its chain, by the definition
of $g_k$.  For a $j$-set $Y$ with $j<k$, by induction on $k-j$, the label
$g_j(Y)$ stays the same or moves upward in $Q_{k-j+1}$, because the label is
defined to be the downset generated by the labels of the precursors.  The
precursors remain the same (except that precursors can be added when $j=k-1$).
By the induction hypothesis, the labels of the precursors stay the same or move
up.  Hence the downset they generate stays the same or becomes larger, which
means that $g_j(Y)$ stays the same or moves up.

After the first $k-2$ vertices and before the last, there are
$|Q_k|-k+2$ vertices, and their labels are initially (and hence always)
above the bottom $k-2$ elements of $Q_k$.  If $\Lmb_k$ is not the label of
any of them, then their labels are confined to a set of $|Q_k|-k+1$ elements in
$|Q_k|$.  By the pigeonhole principle, two of these vertices have the same
label.  We claim that in this situation Builder can make a vertex label go up
in $Q_k$.

Builder picks two vertices $x$ and $y$ having the same label, with $x$ before
$y$.  Since $x$ and $y$ have the same label, Lemma \ref{eq:star} guarantees
that $y$ does not follow $x$.  Builder plays edges to make $y$ follow $x$.
Since labels that change can only move up, Lemma \ref{eq:star} implies that
playing edges to make $y$ follow $x$ causes the label of $y$ to increase in
$Q_k$.

In order to make $y$ follow $x$, we consider an instance of $U(\{x\})$.  For
each leaf in $U(\{x\})$, the associated edge $Z$ ends with $x$.  By
Lemma~\ref{followtree}, $y$ follows $x$ if $Z^+\cup\{y\}$ is an edge for each
such edge $Z$.  Builder plays all such $k$-sets that are not already edges. 

The number of edges played by Builder to make $y$ follow $x$ is at most the
number of leaves in $U(\{x\})$.  Since the children in $U(\{x\})$ of each label
in $Q_j$ form an antichain in $Q_{j-1}$, the number of leaves is bounded
by $\PE i1{k-1} a_i$, where $a_i$ is the maximum size of an antichain in $Q_i$.

As long as no monotone tight path with $m$ edges is created, the labels of the
$|Q_k|-k+1$ vertices we are considering can rise at most $|Q_{k-1}|-k$ times
without reaching $\Lmb_k$, since $\Lmb_k$ is the full downset of size
$|Q_{k-1}|$ in $Q_{k-1}$, and each of these labels initially is the unique
downset of size $k-1$.  Hence
$$1+[(|Q_k|-k+1)(|Q_{k-1}|-k)+1]\PE i1{k-1}a_i$$
moves suffice for Builder to finish the game.  Thus
$\olsrt(\prk)\le |Q_k|\cdot |Q_{k-1}|\PE i1{k-1}a_i$.

The remainder of the proof, obtaining an upper bound on $\olsrt(\prk)$ in terms
of $|Q_k|$ alone, is purely numerical.  Consider any small positive constant
$\epsilon$.  We seek
\begin{equation}\label{eq:Qbounds}
|Q_{2}| a_2 a_1 \le (\lg|Q_3|)^{2+\FR1{t-1}+\eps}
\qquad\text{and}\qquad
|Q_{k-1}|\PE i1{k-1}a_i \le (\lg|Q_{k}|)^{2+\eps}
\quad
\text{(for $k\ge 4$)}.
\end{equation}
We will find positive constants $t_0$ and $m_0$ in terms of $\epsilon$ such
that \eqref{eq:Qbounds} holds when $tm\ge t_0m_0$.

Let $q_i=|Q_i|$.  The rank of an element of $Q_i$ is its size as a downset in
$Q_{i-1}$; hence $Q_i$ has $|Q_{i-1}|+1$ ranks.  Since the minimal and maximal
elements are unique, $Q_i$ has a decomposition into the fewest chains such that
no chain meets all ranks.  Dilworth's Theorem~\cite{Dil} and the pigeonhole
principle then yield $a_i\ge q_i/q_{i-1}$, and hence $a_i\le q_i\le a_iq_{i-1}$.
Since the subsets of a largest antichain in $Q_{i}$ generate distinct downsets,
$q_{i+1}\ge 2^{a_i}$, so $a_i\le\lg q_{i+1}$.  To bound $q_{k-1}\PE i1{k-1}a_i$
in terms of $q_k$, we need $q_i$ to grow rapidly with $i$.  Already we have
$q_{i+1}\ge q_i/q_{i-1}$, but we need better.

Consider first $k=3$.  The computation we use to prove the first part of
\eqref{eq:Qbounds} is
$$
q_2a_2a_1 = tm^t a_2 \le a_2\left(\frac{m^{t-1}}{2\sqrt t}\right)^{t/(t-1)+\eps}
\le (\lg q_3)^{2+\FR1{t-1}+\eps}.
$$
The first step is from $a_1=t$ and $q_2=m^t$.  For the rightmost inequality, we
noted $a_2\le \lg q_3$ above, and Theorem~\ref{MSthm} gives
$m^{t-1}/2\sqrt t\le\lg q_3$.  The middle inequality reduces to
$t(2\sqrt t)^{t/(t-1)+\epsilon}\le m^{\epsilon(t-1)}$.
When $m\ge 4^{1+2/\epsilon}$, this holds for $t\ge2$.
When $(t-1)/\lg t \ge .5+2/\epsilon$, it holds for $m\ge2$.
Hence if we let $m_0= 4^{1+2/\epsilon}$ and let $t_0$ be the solution to
$(t-1)/\lg t=.5+2/\epsilon$, the inequality will hold whenever
$tm\ge t_0m_0$, since that yields $t\ge t_0$ or $m\ge m_0$ when $t,m\ge2$.


In order to prove the inequality of~\eqref{eq:Qbounds} for $k\ge4$, it suffices
to prove
\begin{equation}\label{stronger}
\PE i1{k-1} q_i \le (\lg q_{k})^{1+\eps/2},
\end{equation}
because $a_i\le q_i$ implies $q_{k-1}\PE i1{k-1} a_i< (\PE i1{k-1} q_i)^2$.
In the induction step, we use $1+\epsilon/2<4$ to weaken the induction
hypothesis, proving that $\PE i1{k-2} q_i \le (\lg q_{k-1})^4$ implies
\eqref{stronger}.  As a base step to start the induction, we prove the weaker
statement for $k=3$.  The computation for this is
$$q_2q_1=tm^{t+1} \le (m^{t-1}/t)^4 =(q_2/q_1)^4  \le a_2^4 \le (\lg q_3)^4,$$
in which the only step needing further explanation is
$tm^{t+1}\le (m^{t-1}/t)^4$, which simplifies to $(mt)^5\le m^{3t}$.
This holds when $t=2$ and $m\ge32$, or when $t\ge3$ and $m\ge4$.  It does not
hold when $t=m=3$, but the desired inequality $tm^{t+1}\le (q_2/q_1)^4$ does
hold then.  In any case, we obtain the desired inequality when $tm\ge64$.

For the induction step, we first use $q_i\le a_iq_{i-1}$, the induction
hypothesis, and the fact that $q_{k-1}$ (which exceeds $t(m-1)$) is
sufficiently large to compute
$$
\PE i1{k-1} q_i \le  a_{k-1} q_{k-2} \PE i1{k-2} q_i
< a_{k-1} \left(\PE i1{k-2} q_i\right)^2 \le  \lg q_k  (\lg q_{k-1})^8
\le  q_{k-1}^{\eps/3} \lg q_k.
$$
Now let $\beta=\PE i1{k-1} q_i$.  We weaken
$\beta\le q_{k-1}^{\epsilon/3}\lg q_k$ to $\beta\le\beta^{\epsilon/3}\lg q_k$.
Rearranging to a bound on $\beta$ now yields
$\beta\le (\lg q_k)^{1/(1-\eps/3)}\le (\lg q_k)^{1+\epsilon/2}$,
which completes the proof of \eqref{stronger} and the theorem.
\end{proof}

The argument for the lower bound, presented next, is easier.

\begin{thm}\label{lowerbd}
With $r>k$ and $m=r-k+1$, we have $\olsrt(\prk)\ge |Q_k|/(k\lg |Q_k|)$.
\end{thm}
\begin{proof}
With $\VEC Q1k$ defined as before, we give a strategy for Painter.  Painter
assigns labels to all $j$-sets of vertices that have been played, for
$1\le j\le k$; these labels remain unchanged throughout the game.  The label
$f_j(Y)$ assigned to a $j$-set $Y$ is in $Q_{k-j+1}$.  Since the label of 
a $k$-set is in $Q_1$, it specifies the color to be used on the set if
Builder plays it as an edge.

Let $A$ be a maximum-sized antichain in $Q_k$.  We have noted that 
$|A|\ge |Q_k|/\lg|Q_k|$.
When Builder uses new vertices, Painter gives them distinct unused elements of
$A$ as labels.  Painter will use these labels to avoid making a monochromatic
monotone copy of $\prk$.  Hence Painter can survive for at least $|A|/k$ edges.

In defining labels, the property we will need is that if $Y_1$ and $Y_2$ are
$j$-sets such that $Y_1^+=Y_2^-$ (or equivalently that $Y_1=Y^-$ and $Y_2=Y^+$
for some $(j+1)$-set $Y$), then $f_j(Y_1)\ngeq f_j(Y_2)$.  For $j=1$, the
labels of vertices are chosen as incomparable elements in $Q_k$, so this
holds by construction no matter what order Builder uses to introduce vertices.

For $1\le j\le k-1$, we define $f_{j+1}$ from $f_{j}$ (Builder defined $g_j$
from $g_{j+1}$ in the upper bound).  Given a $(j+1)$-set $Y$, consider $Y^-$
and $Y^+$.  Since $(Y^-)^+=(Y^+)^-$, we are given $f_j$ defined so that
$f_j(Y^-)\ngeq f_j(Y^+)$.  Hence some element of $f_j(Y^+)$ is not in
$f_j(Y^-)$ (as downsets in $Q_{k-j}$).  Painter chooses any such element as the
label $f_{j+1}(Y)$.

Now consider $(j+1)$-sets $Y_1$ and $Y_2$ with $Y_1^+=Y_2^-$.
Both $f_j(Y_2^+)$ and $f_j(Y_2^-)$ are downsets in $Q_{k-j}$, and we chose
$f_{j+1}(Y_2)\in f_j(Y_2^+)-f_j(Y_2^-)$.  Hence the element $f_{j+1}(Y_2)$ is
not below anything in the downset $f_j(Y_2^-)$, including
$f_{j+1}(Y_1)\in f_j(Y_1^+)=f_j(Y_2^-)$.
This means $f_{j+1}(Y_1)\ngeq f_{j+1}(Y_2)$, as needed for the process to
continue.


We have now defined labels for all sets of at most $k$ vertices.
The labels of $k$-sets lie in $Q_1$ and hence are colors with heights.  When
Builder plays a $k$-set, the color used by Painter is the color in its label.
When edges $Y_1$ and $Y_2$ are consecutive in a monotone tight path in color
$i$, so $Y_1^+=Y_2^-$, the property $f_k(Y_1)\ngeq f_k(Y_2)$ implies that the
height of the label in $Q_1$ strictly increases.  Since the chains in $Q_1$
have only $m-1$ elements, no monochromatic monotone copy of $\prk$ occurs.
\end{proof}

%
%

We restrict vertex labels to an antichain in $Q_k$ because Builder has the
power to introduce new vertices between old vertices, and when vertex $x$ is to
the left of vertex $y$ Painter needs to find an element in the label of $y$
that is not in the label of $x$.  If the vertices were known in advance,
then the vertex Ramsey result $R_t(\prk)=|Q_k|+1$ would already allow Painter
to survive $|Q_k|/k$ edges in the on-line game.  On the other hand, our
arguments also yield this result.

\begin{cor}
[Moshkovitz and Shapira~\cite{MS}]
$R_t(\prk)=|Q_k|+1$.
\end{cor}
\begin{proof}
When all vertices are known in advance, or when Builder is constrained to add
vertices only at the high (i.e., right) end (as in the game studied by Fox et
al.~\cite{FPSS}), Painter can use all of $Q_k$ as vertex labels, assigning them
according to a linear extension, level by level.  The initialization
$f_1(\{x\})\ngeq f_1(\{y\})$ for any vertices $x$ and $y$ with $x$ before $y$
then holds.  The rest of the proof is exactly the same, yielding a lower bound
of $|Q_k|/k$ for their game and requiring more than $|Q_k|$ vertices to be
played to force a monochromatic copy of $\prk$.

Since the off-line situation is weaker for Builder, we must work harder for the
upper bound.  All the edges of $\CHnk$ will be played, with $n=|Q_k|+1$.
Painter knows that.  If there is a $t$-coloring that avoids $\prk$, then
Painter can prepare to play that coloring, no matter in what order we add
the edges.  We can allow the labels to be defined as in the on-line game
as we add edges.

Initially, the labels are as at the start of the on-line game, as described
in Remark~\ref{bottom}.  We imagine playing all the edges on the first
$|Q_k|$ vertices first.  If $\Lmb_k$ appears as a label on a vertex, then
as observed in the proof of Theorem~\ref{main} there is an edge using the
last vertex that when added forces $\prk$.  If $\Lmb_k$ does not appear,
then among the first $|Q_k|$ vertices there are vertices $x$ and $y$ (with $y$
later than $x$) having the same labels.  Lemma~\ref{eq:star} as edges are
processed maintains that two vertices cannot have the same label when one
follows the other.  Lemma~\ref{followtree} guarantees that when all the edges
are processed, all the edges that need to be played to make $y$ follow $x$ have
been played.  Hence such $x$ and $y$ cannot exist, and $\Lmb_k$ must occur
as a label on a vertex.
\end{proof}

 
Generalizing these results to $\ell$-loose $k$-uniform monotone paths is
straightforward.  The off-line value $R_t(\prkl)$ was obtained by Cox and
Stolee~\cite{CS}.  The key point is that edges whose last vertices differ by
less than $\ell$ cannot belong to a common $\ell$-loose $k$-uniform monotone
path.  Recall that explicit bounds on $|Q_h| \cdot |Q_{h-1}| \PE i1{h-1}a_i$ in
terms of $|Q_h|$ and $|Q_{h-1}|$ are given in Theorem~\ref{main}.

\begin{thm}
Given $k,\ell,m,t\in\NN$ with $t,m\ge2$ and $\ell\in[k]$, let $r=k+\ell(m-1)$.
Also let $h=\CL{k/\ell}$ and $s=k-(h-1)\ell$.  With $Q_j$ defined in terms of
$k,r,t$ as in the introduction, $R_t(\prkl)=\ell|Q_h|+s$.
Moreover, if $\ell<k$ then
$|Q_h|/k\lg|Q_h|\le\olsrt(\prkl)\le |Q_h|\cdot |Q_{h-1}| \PE i1{h-1}a_i$,
where $a_i$ denotes the size of the largest antichain in $Q_i$, while if $\ell=k$ then $|Q_1|/k\lg |Q_1|\le\olsrt(\prkl)\le |Q_1|+1$.
\end{thm}
\begin{proof}
(Sketch) The value $\ell$ is the {\it shift}; in an $\ell$-loose
$k$-uniform monotone path, it is the number of vertices at the beginning of an
edge that are not included in the next edge.

Let $Y^-$ and $Y^+$ be obtained from a set $Y$ with $|Y|>\ell$ by deleting the
last $\ell$ and the first $\ell$ elements, respectively.  Note that $s$ is the
unique member of $[\ell]$ congruent to $k$ modulo $\ell$.  Given $j$ with
$1\le j\le h$, let $j'=k-(h-j)\ell$; the values of $j'$
are $\{i\in[k]\st i\equiv k\mod \ell\}$.

\medskip
{\bf Lower Bound} {\it (Painter strategy):}
Painter will assign labels to subsets of the vertices whose size is congruent
to $k$ modulo $\ell$.  In particular, the label $f_j(Y)$ will be in $Q_{h-j+1}$
for each $j'$-set $Y$ of vertices.  As noted earlier, in $Q_h$ there is an
antichain of size at least $|Q_h|/|Q_{h-1}|$.  Painter initially fixes a
largest antichain $A$ in $Q_h$ and uses distinct elements of $A$ to name the
vertices as they are introduced by Builder; we do not call these ``labels'' in
the sense used earlier.  The smallest sets given labels by Painter have size
$s$.  For each $s$-set $Y$, let $f_1(Y)$ be the element of $A$ that Painter
used to name its rightmost vertex.


For $1\le j\le h$, again we need $f_j(Y_1)\not\ge f_j(Y_2)$ for $j'$-sets $Y_1$
and $Y_2$ such that there exists $Y$ with $Y_1=Y^-$ and $Y_2=Y^+$.  Note that
such a set $Y$ may be introduced after later moves by Builder's introduction
of new vertices.  However, if $Y_1$ and $Y_2$ have the same highest vertex,
then this can never occur, and Painter can have the same label on $Y_1$ and
$Y_2$.

For $1\le j\le h-1$, define $f_{j+1}$ from $f_j$ by letting $f_{j+1}(Y)$ be any
element of $f_j(Y^+)$ not in $f_j(Y^-)$.  The inductive proof of the needed
property $f_j(Y_1)\not\ge f_j(Y_2)$ is the same as in Theorem~\ref{lowerbd}.
The Painter strategy is as defined there: the resulting labels of $k$-sets
under $f_k$ lie in $Q_1$, and the color used by Painter on an edge played by
Builder is the color of the chain containing its label.  Since heights must
strictly increase along $\ell$-loose $k$-uniform paths, no monochromatic copy
of $\prkl$ occurs.  Painter can survive any $a_h/k$ edges, where $a_h=|A|$.

In a restricted version of the game where Builder must add vertices in order
from low to high, or where the vertices are specified in advance, Painter can
use all elements of $Q_h$ as vertex names (in the order of a linear extension
of $Q_h$).  Furthermore, Painter can then use the same name on $\ell$
consecutive vertices, since edges whose highest vertices differ by less than
$\ell$ cannot belong to the same copy of $\prkl$, and no vertices will be
inserted between two already having names.  In addition, the first $s-1$
vertices receive no names from $Q_h$, since the smallest sets needing labels
have size $s$.  Again the process proceeds: $s$-sets receive as label the
element of $Q_h$ assigned to their highest vertex.  Note that if
$|\max Y_2 - \max Y_1| < \ell$, then $Y_1$ and $Y_2$ can never be extended
leftward to edges in the same copy of $\prkl$.  In this way, Painter can
survive $\ell|Q_h|+s-1$ vertices.  Hence $R_t(\prkl)\ge\ell|Q_h|+s$,
as in~\cite{CS}.

\medskip
{\bf Upper Bound} {\it (Builder strategy):} Builder uses
$\ell|Q_h|+s$ vertices, assigning labels to sets whose size is congruent to $k$
modulo $\ell$, down to size $s$.  Actually, Builder assigns labels only to sets
whose last $s$ vertices are consecutive, called {\it basic} sets; Builder also
plays only basic edges.  Henceforth consider only basic sets.  Note that there
are $\ell|Q_h|+1$ basic sets of size $s$.

Builder assigns a label in $Q_1$ to edges and a label in $Q_{h-j+1}$ to the
sets of size $j'$ for $h> j\ge 1$ (note that $j'=j$ when $\ell=1$).  For an
edge $Y$ with color $i$ in $G$, the label $g_h(Y)$ is the element of height
$p$ on the $i$th chain in $Q_1$, where $p$ is the number of edges in the
longest $\ell$-loose $k$-uniform monotone path with last edge $Y$ in the
current colored hypergraph.  For $h>j\ge1$, the {\it precursors} of a $j'$-set
$Y$ are the $(j'+\ell)$-sets obtained by adding $\ell$ elements to $Y$ that are
smaller than the least element of $Y$; that is, the precursors are the sets $Z$
such that $Z^+=Y$.

With these generalizations of earlier definitions, the definitions of $g_j$ for
$1\le j<h$ and the relation of ``follows'' are the same as in
Definitions~\ref{labelfcn} and \ref{follows}.
In particular, note that if $Y_2$ follows $Y_1$, then the rightmost element of
$Y_2$ must be at least $\ell$ positions to the right of the rightmost element
of $Y_1$.  The statement and proof of Lemma~\ref{eq:star} are the same, except
that $g_k$ and $Q_{k-j+1}$ generalize to $g_h$ and $Q_{h-j+1}$, and $\CH{[n]}j$
becomes $\CH{[n]}{j'}$.  In Definition~\ref{UY} and Lemma~\ref{followtree} we
generalize $j$-set and $(j+1)$-set to basic $j'$-set and basic $(j'+\ell)$-set,
and again $k$ generalizes to $h$ in various subscripts.

Now Remark~\ref{bottom} and Theorem~\ref{main} also generalize naturally to
yield
$\olsrt(\prkl)\le |Q_h| \cdot |Q_{h-1}| \PE i1{h-1}a_i$ for $\ell<k$
(or equivalently $h\ge2$).  
Note that the labels $\bz_h^0,\ldots,\bz_h^{h-2}$ of the chain at the bottom of
$Q_h$ are assigned to the first $(h-1)\ell$ basic sets of size $s$, where each
label is used on $\ell$ consecutive sets.  (Since each basic $s$-set
is an interval of $s$ consecutive vertices, these sets form a order with the
next basic $s$-set shifting by one from the previous one.) 
For $0\le i\le (h-1)\ell -1$, the set $[i+1,i+s] \in \binom{[n]}{s}$ is
assigned label $\bz_h^{\FL{i/\ell}}$.  These labels never change, since no edge
can be played ending at one of these sets.


The labels of the basic $s$-sets after the first $(h-2)\ell$ are confined to
$|Q_h|-h+1$ labels in $Q_h$ (as long as none of them becomes $\Lmb_h$).
Among those, Builder will focus on basic $s$-sets of the form
$[i\ell+1,i\ell+s]$ for $h-2 \le i\le |Q_h|-1$, which we call {\it restricted}
basic $s$-sets.  Since there are $|Q_h|-h+2$ of these sets, when Builder is
ready to move the pigeonhole principle guarantees that some label in $Q_h$ is
assigned to at least two restricted basic $s$-sets.  This guarantees the
existence of two basic $s$-sets $X$ and $Y$ with the same label whose rightmost
vertices differ by at least $\ell$.  By the generalization of
Lemma~\ref{eq:star}, $Y$ does not follow $X$.  Builder can then play edges as
guaranteed by the generalization of Lemma~\ref{followtree} to make $Y$ follow
$X$, which as in Theorem~\ref{main} makes the label of $Y$ go up.  A label can
increase at most $|Q_{h-1}|-h$ times before reaching $\Lmb_h$.


Hence Builder can play to force an $s$-set $Z$ with label $\Lmb_h$ ending
before the last $\ell$ vertices.  As in Theorem~\ref{main}, some $(k-\ell)$-set
$Y$ ending with $Z$ will then have label $(m-1,\dots,m-1)$, the top element of 
$Q_2$.  By playing the $k$-set consisting of $Y$ and the last $\ell$ vertices,
Builder wins.


Since in fact the label of the leftmost restricted basic $s$-set never
changes, the number of edges played is at most
\[
1+[ (|Q_h|-h+1)(|Q_{h-1}|-h) +1]\PE i1{h-1}a_i,
\]
which for $h\ge2$ is at most
$|Q_h| \cdot |Q_{h-1}| \PE i1{h-1}a_i$.
Note, however, that since Builder used only $\ell|Q_h|+s$ vertices, we have
$R_t(\prkl)=\ell|Q_h|+s$.
In the case $h=1$ (that is, $\ell=s=k$), Builder simply plays the basic
edges (intervals) $[ik+1,(i+1)k]$ for $0\le i\le |Q_1|$.  Since
$[i'k+1,(i'+1)k]$ follows $[ik+1,(i+1)k]$ whenever $i<i'$, Painter is forced
to use distinct labels on the edges and loses.  This gives the desired upper
bounds on $R_t(\prkl)$ and $\olsrt(\prkl)$ for $\ell=k$.
\end{proof}

\section{Directed Graphs}

The ordered Ramsey problem can be described using directed graphs and
hypergraphs.  An orientation of an edge is a permutation of its vertices.
An ordered hypergraph can be viewed as a directed hypergraph in which the
orientation of each edge is the permutation inherited from the vertex ordering.
In particular, an ordered tight path is a directed hypergraph in which the
edges are the $k$-sets of consecutive vertices, oriented in increasing order in
each edge.  In a general $k$-uniform directed hypergraph, $k$-sets may appear
up to $k!$ times, once with each orientation.

When Builder has the power to play edges of a general directed hypergraph
in seeking to force a monochromatic directed tight path, Painter can follow
a strategy like that above, using an antichain in $Q_k$ for vertex labels.
All oriented $j$-tuples must be labeled, for $1\le j\le k$, so the lower bound
will be $|Q_k|/(k!\lg |Q_k|)$.

Let us consider this problem in the off-line setting for $k=2$.
Hence we are seeking the size Ramsey number of the directed path $P_{m+1}$
in the model where arbitrary host digraphs are allowed.  The trivial upper
bound is again $\CH{m^t+1}2$, achieved by playing increasing edges for all
pairs on $R_t(P_{m+1})$ vertices in the ordered setting.  For the off-line
model, Builder is weaker, and we obtain a better lower bound than for the
on-line game.

\begin{thm}
In the setting of directed graphs, $\srt(P_{m+1})\ge\CH{|B|+1}2$,
where $B$ is the family of elements in $M^t$ with sum $\FL{(m-1)t/2}$.
\end{thm}

\begin{proof}
A graph with fewer than $\binom{|B|+1}{2}$ edges is $(|B|-1)$-degenerate and
hence $|B|$-colorable.  Hence we may suppose that the underlying undirected
graph of the host digraph is $|B|$-colorable.  Painter specifies a proper
vertex coloring whose colors correspond to the elements of $B$.  Each vertex
$v$ has a label $a(v)\in B$, and adjacent vertices always have distinct labels.
As in Theorem~\ref{k=2}, Painter can choose for each (directed) edge $uv$ a
color $i$ such that $a_i(v)>a_i(u)$.  Again at every vertex $w$ the length of
any path in color $i$ reaching $w$ is at most $a_i(w)$, since the $i$th
coordinate strictly increases along paths whose edges have color $i$.
\end{proof}

The off-line size Ramsey problem for paths in digraphs (with $t=2$) was also
studied by Ben-Eliezer, Krivelevich, and Sudakov~\cite{BKS}.
They considered both when
Builder can present only oriented graphs (no $2$-cycles) and when Builder can
present any digraph, yielding size Ramsey numbers $S_{ori}$ and $S_{dir}$
respectively.  Note that $S_{dir}\le S_{ori}$ when the parameters are equal.

For the general digraph model, which we considered above,
the arguments of~\cite{BKS} yield the following bounds:
\[
\left(\frac{m+1}{3t-3}\right)^{2t-2} \le S_{dir} \le  4(m+1)^{2t-2}.
\]
Since they focus on constant $t$, they state the result as
$S_{dir} = \Theta(m^{2t-2})$.
Since
$|B|\ge \frac23 m^{t-1}/\sqrt t$, our lower bound strengthens theirs.

Their lower bound for $S_{ori}$ is higher than their upper bound for $S_{dir}$
(Bucic, Letzter, and Sudakov~\cite{BLS} improved their upper bound on
$S_{ori}$).
They prove
\[ 
C_1(t)
\frac{(m+1)^{2t-2}(\log(m+1))^{\frac{1}{t-1}}}{(\log\log(m+1))^{\frac{t+1}{t-1}}}
\le S_{ori} \le C_2 (m+1)^{2t-2}(\log(m+1))^2
\]
where $C_2$ is an absolute constant, but $C_1(t)$ depends on $t$.
They require
\[
C_1(t) < \frac{C^{1/(t-1)}}{8(2t-2)^{t-1}(16(t-1)^2)^t}
\]
for some absolute constant $C$.
Therefore, their lower bound is at most
\[
\frac{1}{ (2t)^{3t}}
\frac{(m+1)^{2t-2}(\log(m+1))^{\frac{1}{t-1}}}{(\log\log(m+1))^{\frac{t+1}{t-1}}},
\]
which remains smaller than ours when $t$ grows faster than $\sqrt{\log\log m}$.



\end{document}